\documentclass[article]{IEEEtran}

\usepackage{amsmath,amssymb,amsfonts,amscd}%
\usepackage{theorem}%
\usepackage{dsfont}%
\usepackage{graphicx}%
\usepackage{color}%
\usepackage{hyperref}%

\allowdisplaybreaks

\theoremstyle{change}%

\newtheorem{definition}{Definition:}[section]%
\newtheorem{theorem}[definition]{Theorem:}%
\newtheorem{lemma}[definition]{Lemma:}%
\newtheorem{corollary}[definition]{Corollary:}%
{\theorembodyfont{\rmfamily} \newtheorem{remark}[definition]{Remark:}}%
{\theorembodyfont{\rmfamily} \newtheorem{example}[definition]{Example:}}%

\newcommand{\tm}{\times}%
\newcommand{\tp}{\mathrm{top}}%
\newcommand{\N}{\mathbb{N}}%
\newcommand{\Z}{\mathbb{Z}}%
\newcommand{\R}{\mathbb{R}}%
\newcommand{\supp}{\mathrm{supp}}%
\newcommand{\ep}{\varepsilon}%
\newcommand{\spn}{\mathrm{span}}%
\newcommand{\sep}{\mathrm{sep}}%
\newcommand{\rmd}{\mathrm{d}}%
\newcommand{\rme}{\mathrm{e}}%
\newcommand{\rmD}{\mathrm{D}}%
\newcommand{\cl}{\mathrm{cl}}%
\newcommand{\unit}{\mathds{1}}%

\newcommand{\AC}{\mathcal{A}}%
\newcommand{\BC}{\mathcal{B}}%
\newcommand{\EC}{\mathcal{E}}%
\newcommand{\FC}{\mathcal{F}}%
\newcommand{\GC}{\mathcal{G}}%
\newcommand{\IC}{\mathcal{I}}%
\newcommand{\MC}{\mathcal{M}}%
\newcommand{\TC}{\mathcal{T}}%

\newcommand{\T}{\mathbb{T}}%

\IEEEoverridecommandlockouts

\allowdisplaybreaks

\begin{document}

\title{Metric and topological entropy bounds for optimal coding of stochastic dynamical systems}

\author{Christoph Kawan and Serdar Y\"{u}ksel\thanks{C.~Kawan is with the Faculty of Computer Science and Mathematics, University of Passau, 94032 Passau, Germany (e-mail: christoph.kawan@uni-passau.de). S.~Y\"uksel is with the Department of Mathematics and Statistics, Queen's University, Kingston, Ontario, Canada, K7L 3N6 (e-mail: yuksel@mast.queensu.ca). This research was supported in part by the Natural Sciences and Engineering Research Council (NSERC) of Canada. Some results of this paper appeared in part at the 2017 IEEE International Symposium on Information Theory.}}%
\date{}%
\maketitle%

\begin{abstract}
We consider the problem of optimal zero-delay coding and estimation of a stochastic dynamical system over a noisy communication channel under three estimation criteria concerned with the low-distortion regime. The criteria considered are (i) a strong and (ii) a weak form of almost sure stability of the estimation error as well as (ii) quadratic stability in expectation. For all three objectives, we derive lower bounds on the smallest channel capacity $C_0$ above which the objective can be achieved with an arbitrarily small error. We first obtain bounds through a dynamical systems approach by constructing an infinite-dimensional dynamical system and relating the capacity with the topological and the metric entropy of this dynamical system. We also consider information-theoretic and probability-theoretic approaches to address the different criteria. Finally, we prove that a memoryless noisy channel in general constitutes no obstruction to asymptotic almost sure state estimation with arbitrarily small errors, when there is no noise in the system. The results provide new solution methods for the criteria introduced (e.g., standard information-theoretic bounds cannot be applied for some of the criteria) and establish further connections between dynamical systems, networked control, and information theory, and especially in the context of nonlinear stochastic systems.
\end{abstract}

\section{Introduction}

In this paper, we consider nonlinear stochastic systems given by an equation of the form%
\begin{equation}\label{eq_stochsys}
  x_{t+1} = f(x_t,w_t).%
\end{equation}
Here $x_t$ is the state at time $t$ and $(w_t)_{t\in\Z_+}$ is an i.i.d.~sequence of random variables with common distribution $w_t \sim \nu$, modeling the noise. In general, we assume that%
\begin{equation*}
  f:X \tm W \rightarrow X%
\end{equation*}
is a Borel measurable map, where $(X,d)$ is a complete metric space and $W$ a measurable space, so that for any $w\in W$ the map $f(\cdot,w)$ is a homeomorphism of $X$. We further assume that $x_0$ is a random variable on $X$ with an associated probability measure $\pi_0$, stochastically independent of $(w_t)_{t\in\Z_+}$. We use the notations%
\begin{equation*}
 f_w(x) = f(x,w),\qquad  f^x(w) = f(x,w).%
\end{equation*}
so that  $f_w:X \rightarrow X$ and $f^x:W \rightarrow X$.

The system \eqref{eq_stochsys} is connected over a possibly noisy channel with a finite capacity to an estimator, as shown in Fig.~\ref{LLL1}. The estimator has access to the information it has received through the channel. A source coder maps the source symbols (i.e., state values) to corresponding channel inputs. The channel inputs are transmitted through the channel; we assume that the channel is a discrete channel with input alphabet $\MC$ and output alphabet $\MC'$.%

\begin{figure}[htbp]
\begin{center}
\includegraphics[width=8.0cm,height=1.5cm,angle=0]{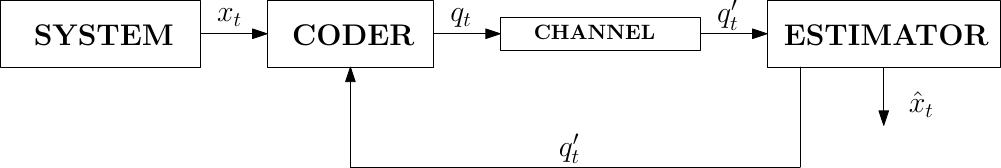}
\end{center}
\caption{Coding and state estimation over a noisy channel with feedback \label{LLL1}}
\end{figure}

We refer by a {\em coding policy} $\Pi$, to a sequence of functions $(\gamma^e_t)_{t\in\Z_+}$ which are causal such that the channel input at time $t$, $q_t \in \MC$, under $\Pi$ is generated by a function of its local information, i.e.,%
\begin{equation*}
  q_t = \gamma^e_t(\IC^e_t),%
\end{equation*}
where $\IC^e_t=\{x_{[0,t]}, q'_{[0,t-1]}\}$ and $q_t \in \MC$, the channel input alphabet given by $\MC = \{1,2,\ldots,M\},$ for $0 \leq t \leq T-1$. Here, we use the notation $x_{[0,t-1]} = \{x_s : 0 \leq s \leq t-1 \}$ for $t \geq 1$.%

The channel maps $q_t$ to $q'_t$ in a stochastic fashion so that $P(q'_t|q_t,q_{[0,t-1]},q'_{[0,t-1]})$ is a conditional probability measure on $\MC'$ for all $t \in \Z_+$. If this expression is equal to $P(q'_t|q_t)$, the channel is said to be memoryless, i.e., the past variables do not affect the channel output $q'_t$ given the current channel input $q_t$.%

The receiver, upon receiving the information from the channel, generates an estimate $\hat{x}_t$ at time $t$, also causally: An admissible causal {\em estimation policy} is a sequence of functions $(\gamma^d_t)_{t\in\Z_+}$ such that $\hat{x}_t = \gamma^d_t(q'_{[0,t]})$ with%
\begin{equation*}
  \gamma^d_t:(\MC')^{t+1} \to X, \quad t \geq 0.%
\end{equation*}

For a given $\ep>0$, we denote by $C_{\ep}$ the smallest channel capacity above which there exist an encoder and an estimator so that one of the following estimation objectives is achieved:%
\begin{enumerate}
\item[(E1)] Eventual almost sure stability of the estimation error: There exists $T(\ep)\geq0$ so that%
\begin{equation}\label{eq_BowenBall_obj}
  \sup_{t \geq T(\ep)} d(x_t,\hat{x}_t) \leq \ep \mbox{\quad a.s.}%
\end{equation}
\item[(E2)] Asymptotic almost sure stability of the estimation error:%
\begin{equation}\label{eq_AlmostSureBowenBall_obj}
  P\bigl(\limsup_{t \to \infty} d(x_t,\hat{x}_t) \leq \ep\bigr) = 1.%
\end{equation}
\item[(E3)] Asymptotic quadratic stability of the estimation error in expectation:%
\begin{equation}\label{eq_quadratic_obj}
  \limsup_{t \to \infty} E[d(x_t,\hat{x}_t)^2] \leq \ep.%
\end{equation}
\end{enumerate}

\subsection{Literature Review and Contributions}

In a recent work \cite{KawanYukselITarXiv}, we investigated the same problem for the special case involving only deterministic systems and discrete noiseless channels. In this paper, we will provide further connections between the ergodic theory of dynamical systems and information theory by answering the problems posed in the previous section and relating the answers to the concepts of either metric or topological entropy. Our findings complement and generalize our results in \cite{KawanYukselITarXiv} since here we consider stochasticity in the system dynamics and/or the communication channels.%

As we note in \cite{KawanYukselITarXiv}, optimal coding of stochastic processes is a problem that has been studied extensively; in information theory in the context of per-symbol cost minimization, in dynamical systems in the context of identifying representational and equivalence properties between dynamical systems, and in networked control in the context of identifying information transmission requirements for stochastic stability or cost minimization. As such, for the criteria laid out in (E1)-(E3) above, the results in our paper are related to the efforts in the literature in the following three general areas. %

{\bf Dynamical systems and ergodic theory.} Historically there has been a symbiotic relation between the ergodic theory of dynamical systems and information theory (see, e.g., \cite{ShieldsIT,Downarowicz} for comprehensive reviews). Information-theoretic tools have been foundational in the study of dynamical systems, for example the metric (also known as Kolmogorov-Sinai or measure-theoretic) entropy is crucial in the celebrated Shannon-McMillan-Breiman theorem as well as two important representation theorems: Ornstein's (isomorphism) theorem and the Krieger's generator theorem \cite{GrayProbabilit,Orn,Ornstein,katok2007fifty,Downarowicz}. The concept of sliding block encoding \cite{GrayIT} is a stationary encoding of a dynamical system defined by the shift process, leading to fundamental results on the existence of stationary codes which perform as good as the limit performance of a sequence of optimal block codes. For topological dynamical systems, the theory of entropy structures and symbolic extensions answers the question to which extent a system can be represented by a symbolic system (under preservation of some topological structure), cf.~\cite{Downarowicz} for an overview of this theory. Entropy concepts have extensive operational practical usage in identifying limits on source and channel coding for a large class of sources \cite{ShieldsIT,GrayIT,GrayNeuhoff}.%


{\bf Networked control and stochastic stability under information constraints.} In networked control, there has been a recurrent interest in identifying limitations on state estimation and control under information constraints. The results in this area have typically involved linear systems, and in the non-linear case the studies have only been on deterministic systems estimated/controlled over deterministic channels, with few exceptions. For linear systems, data-rate theorem type results have been presented in \cite{Brockett,TatikondaThesis,NairEvans,Matveev,MatveevSavkin}.%

The papers \cite{liberzon2016entropy,liberzon2017,pogromsky2011topological,MPo,MP2} studied state estimation for non-linear deterministic systems and noise-free channels. In \cite{liberzon2016entropy,liberzon2017}, Liberzon and Mitra characterized the critical data rate $C_0$ for exponential state estimation with a given exponent $\alpha\geq0$ for a continuous-time system on a compact subset $K$ of its state space. As a measure for $C_0$, they introduced a quantity called estimation entropy $h_{\mathrm{est}}(\alpha,K)$, which equals the topological entropy on $K$ in case $\alpha=0$, but for $\alpha>0$ is no longer a purely topological quantity. The paper \cite{kawan2016state} provided a lower bound on $h_{\mathrm{est}}(\alpha,K)$ in terms of Lyapunov exponents under the assumption that the system preserves a smooth measure. In \cite{MPo,MP2}, Matveev and Pogromsky studied three estimation objectives of increasing strength for discrete-time non-linear systems. For the weakest one, the smallest bit rate was shown to be equal to the topological entropy. For the other ones, general upper and lower bounds were obtained which can be computed directly in terms of the linearized right-hand side of the equation generating the system. 

A further closely related paper is due to Savkin \cite{Savkin06}, which uses topological entropy to study state estimation for a class of non-linear systems over noise-free digital channels. In fact, our results can be seen as stochastic analogues of some of the results presented in \cite{Savkin06}, which show that for sufficiently perturbed deterministic systems state estimation with arbitrarily small error is not possible over finite-capacity channels. See Remark \ref{SavkinRemark} for further discussions with regard to \cite{Savkin06}.

A related problem is the control of non-linear systems over communication channels. This problem has been studied in few publications, and mainly for deterministic systems and/or deterministic channels. Recently, \cite{yuksel2015stability} studied stochastic stability properties for a more general class of stochastic non-linear systems building on information-theoretic bounds and Markov-chain-theoretic constructions. However, these bounds do not distinguish between the unstable and stable components of the tangent space associated with a dynamical non-linear system, while the entropy bounds established in this paper make such a distinction, but only for estimation problems and in the low-distortion regime.%

{\bf Zero-delay coding over communication channels.}
In our setup, we have causality as a restriction in coding and decoding. Zero-delay coding is an increasingly important research area of significant practical relevance, as we review in \cite{KawanYukselITarXiv}. Notable papers include the classical works by Witsenhausen \cite{Witsenhausen}, Walrand \& Varaiya \cite{WalrandVaraiya} and Teneketzis \cite{Teneketzis}. The findings of \cite{WalrandVaraiya} have been generalized to continuous sources in \cite{YukIT2010arXiv} (see also \cite{YukLinZeroDelay} and \cite{BorkarMitterTatikonda}, where the latter imposes a structure apriori); and the structural results on optimal fixed-rate coding in \cite{Witsenhausen} and \cite{WalrandVaraiya} have been shown to be applicable to setups when one also allows for variable-length source coding in \cite{KMe1}. Structural results on coding over noisy channels have been studied in \cite{WalrandVaraiya,Teneketzis,MahTen09,wood2016optimal} among others. Related work also includes \cite{wood2016optimal,MahTen09,AsnaniWeissman,javidi2013dynamic} which have primarily considered the coding of discrete sources. \cite{wood2016optimal,YukLinZeroDelay,AsnaniWeissman,BorkarMitterTatikonda,javidi2013dynamic} have considered infinite horizon problems and, in particular, \cite{wood2016optimal} has established the optimality of stationary and deterministic policies for finite aperiodic and irreducible Markov sources. A related lossy coding procedure was introduced by Neuhoff and Gilbert \cite{NeuhoffGilbert}, called {\it causal source coding}, which has a different operational definition, since delays in coding and decoding are allowed so that efficiencies through entropy coding can be utilized. Further discussions on the literature are available in \cite{KMe1,YukLinZeroDelay,AsnaniWeissman,NayyarTeneketzis}. Among those that are most relevant to our paper is \cite{LinderZamir}, where causal coding under a high rate assumption for stationary sources and individual sequences was studied though only for a source coding context.

The setup with Gaussian channels is a special case studied extensively for the coding of linear systems. We will not consider such a setup in this paper; though we note that explicit results have been obtained for a variety of criteria in the literature.%


{\bf Contributions.} In view of this literature review, we make the following contributions. We establish that for (E1), the topological entropy of a properly defined infinite-dimensional dynamical system defining the stochastic evolution of the process provides lower bounds, for (E2) a lower bound is provided by the metric entropy and for (E3) the metric entropy of this system also provides a lower bound under a restriction on the class of encoders considered. Through a novel analysis which also recovers the widely studied linear case, we also provide achievability results for the case where the only stochasticity is in the communication channel and the initial state. We also establish impossibility results when noise is such that the process is sufficiently mixing. We show that our results reduce to those reported in \cite{KawanYukselITarXiv} for deterministic systems. An implication is that the rate bounds may not be continuously dependent on the presence of stochastic noise, i.e., an arbitrarily small noisy perturbation in the system dynamics may lead to a discontinuous change in the rate requirements for each of the criteria. Throughout the analysis, we provide further connections between information theory and dynamical systems by identifying the operational usage of entropy concepts for the three different estimation criteria.

\section{Preliminaries}\label{sec_prelim}

{\bf Notation:} All logarithms in this paper are taken to the base $2$. By $\N$ we denote the set of positive integers. We write $\Z$ for the set of all integers and $\Z_+ = \N \cup \{0\}$. By $\unit_A$ we denote the characteristic function of a set $A$. We write $B_{\ep}(x)$ for the open ball of radius $\ep>0$ centered at $x\in\R^N$. If $f:X\rightarrow Y$ is a measurable map between measurable spaces $(X,\FC)$ and $(Y,\GC)$, we write $f_*$ for the push-forward operator associated with $f$ on the space of measures on $(X,\FC)$, i.e., for any measure $\mu$ on $(X,\FC)$, $f_*\mu$ is the measure on $(Y,\GC)$ defined by $(f_*\mu)(G) := \mu(f^{-1}(G))$ for all $G\in\GC$.%

\subsection{Entropy notions for dynamical systems}

An important concept used in this paper is the topological entropy of a dynamical system. If $f:X \rightarrow X$ is a continuous map on a metric space $(X,d)$, and $K \subset X$ is a compact set, we say that $E \subset K$ is {\em $(n,\ep;f)$-separated} for some $n\in\N$ and $\ep>0$ if for all $x,y\in E$ with $x\neq y$, $d(f^i(x),f^i(y)) > \ep$ for some $i \in \{0,1,\ldots,n-1\}$. We write $r_{\sep}(n,\ep,K;f)$ for the maximal cardinality of an $(n,\ep;f)$-separated subset of $K$ and define the {\em topological entropy $h_{\tp}(f,K)$ of $f$ on $K$} by%
\begin{eqnarray*}
&&  h_{\sep}(f,\ep,K) := \limsup_{n\rightarrow\infty}\frac{1}{n}\log r_{\sep}(n,\ep,K;f),\nonumber \\
&& h_{\tp}(f,K) := \lim_{\ep\downarrow0}h_{\sep}(f,\ep,K).%
\end{eqnarray*}
If $X$ is compact and $K = X$, we omit the argument $K$ and call $h_{\tp}(f)$ the {\em topological entropy of $f$}. Alternatively, one can define $h_{\tp}(f,K)$ using $(n,\ep)$-spanning sets. A set $F \subset X$ {\em $(n,\ep)$-spans} another set $K \subset X$ if for each $x\in K$ there is $y\in F$ with $d(f^i(x),f^i(y)) \leq \ep$ for $i=0,1,\ldots,n-1$. Letting $r_{\spn}(n,\ep,K;f)$ (or $r_{\spn}(n,\ep,K)$ if the map $f$ is clear from the context) denote the minimal cardinality of a set which $(n,\ep)$-spans $K$, the topological entropy of $f$ on $K$ satisfies%
\begin{equation*}
  h_{\tp}(f,K) = \lim_{\ep\downarrow0}\limsup_{n\rightarrow\infty}\frac{1}{n}\log r_{\spn}(n,\ep,K;f).%
\end{equation*}
If $f:X \rightarrow X$ is a measure-preserving map on a probability space $(\Omega,\FC,\mu)$, i.e., $f_*\mu = \mu$, its {\em metric entropy} $h_{\mu}(f)$ is defined as follows. Let $\AC$ be a finite measurable partition of $X$. Then the entropy of $f$ with respect to $\AC$ is defined by%
\begin{equation}\label{eq_metent}
  h_{\mu}(f;\AC) := \lim_{n\rightarrow\infty}\frac{1}{n}H_{\mu}\Bigl(\bigvee_{i=0}^{n-1}f^{-i}\AC\bigr).%
\end{equation}
Here $\bigvee$ denotes the join operation, i.e., $\bigvee_{i=0}^{n-1}f^{-i}\AC$ is the partition of $X$ consisting of all intersections of the form $A_0 \cap f^{-1}(A_1) \cap \ldots \cap f^{-n+1}(A_{n-1})$ with $A_i \in \AC$. For any partition $\BC$ of $X$, $H_{\mu}(\BC) = -\sum_{B\in\BC} \mu(B)\log \mu(B)$ is the Shannon entropy of $\BC$. The existence of the limit in \eqref{eq_metent} follows from a subadditivity argument. The metric entropy of $f$ is then defined by%
\begin{equation*}
  h_{\mu}(f) := \sup_{\AC}h_{\mu}(f;\AC),%
\end{equation*}
the supremum taken over all finite measurable partitions $\AC$ of $X$. If $f$ is continuous, $X$ is compact metric and $\mu$ is ergodic, there is an alternative characterization of $h_{\mu}(f)$ due to Katok \cite{katok1980}:%

For any $n\in\N$, $\ep>0$ and $\delta \in (0,1)$ put%
\begin{eqnarray*}
&&  r_{\spn}(n,\ep,\delta)  \\
&&\quad  := \min\left\{ r_{\spn}(n,\ep;A) : A \subset X \mbox{ Borel},\ \mu(A) \geq 1-\delta \right\}.%
\end{eqnarray*}

Then for every $\delta \in (0,1)$ it holds that%
\begin{equation*}
  h_{\mu}(f) = \lim_{\ep\downarrow0}\limsup_{n\rightarrow\infty}\frac{1}{n}\log r_{\spn}(n,\ep,\delta).%
\end{equation*}
Topological and metric entropy are related to each other via the variational principle \cite{Mis}: For a continuous map $f:X\rightarrow X$ on a compact metric space $X$,%
\begin{equation*}
  h_{\tp}(f) = \sup_{\mu}h_{\mu}(f),%
\end{equation*}
the supremum taken over all $f$-invariant Borel probability measures $\mu$, i.e., such with $f_*\mu = \mu$.%

If two maps $f:X \rightarrow X$ and $g:Y \rightarrow Y$ on compact metric spaces $X$ and $Y$ satisfy $h \circ f = g \circ h$ with a homeomorphism $h:X \rightarrow Y$, they are called \emph{topologically conjugate} and $h$ is called a \emph{topological conjugacy}. In this case, the topological entropy of $f$ and $g$ is the same, i.e., $h_{\tp}(f) = h_{\tp}(g)$. If $h$ is only a continuous surjection from $X$ to $Y$, then $g$ is called a \emph{topological factor} of $f$ and $h_{\tp}(g) \leq h_{\tp}(f)$.%

\section{A Dynamical Systems Approach}\label{sec_dynsys}


In order to use the concepts of topological and metric entropy, defined for deterministic maps, we associate a shift map with the given stochastic system \eqref{eq_stochsys}. More precisely, we consider the space $X^{\Z_+}$ of all sequences in $X$, equipped with the product topology. We write $\bar{x} = (x_0,x_1,x_2,\ldots)$ for the elements of $X^{\Z_+}$ and we fix the product metric%
\begin{equation}
  D(\bar{x},\bar{y}) := \sum_{t=0}^{\infty}\frac{1}{2^t}\frac{d(x_t,y_t)}{1 + d(x_t,y_t)}, \label{productM}
\end{equation}
where $d(\cdot,\cdot)$ is the given metric on $X$. A natural dynamical system on $X^{\Z_+}$ is the shift map $\theta:X^{\Z_+} \rightarrow X^{\Z_+}$, $(\theta\bar{x})_t \equiv x_{t+1}$, which is continuous with respect to the product topology. An analogous shift map is defined on $W^{\Z_+}$ and denoted by $\vartheta$.%

Observing that the sequence of random variables $(x_t)_{t\in\Z_+}$ forms a Markov chain, when $x_0$ is fixed, the following lemma shows how a stationary measure of this Markov chain defines an invariant measure for $\theta$.%

\begin{lemma}\label{lem_mu}
Let $\pi$ be a stationary measure of the Markov chain $(x_t)_{t\in\Z_+}$. Then an invariant Borel probability measure $\mu$ for $\theta$ is defined on cylinder sets by%
\begin{align*}
  &\mu(B_0 \tm B_1 \tm \cdots \tm B_n \tm X^{[n+1,\infty)})\\
	&:= \int_{B_0 \tm B_1 \tm \cdots \tm B_n}\pi(\rmd x_0) P(\rmd x_1|x_0) \ldots P(\rmd x_n|x_{n-1}),%
\end{align*}
where $B_0,B_1,\ldots,B_n$ are arbitrary Borel sets in $X$. Here%
\begin{eqnarray*}
&&  P(x_{n+1} \in B | x_n = x) = P(f(x_n,w) \in B | x_n=x) \\
&&   = \nu((f^x)^{-1}(B)).%
\end{eqnarray*}
The support of $\mu$ is contained in the closure of the set of all trajectories, i.e.,
\[\supp\mu \subset \cl \TC \]
with
\begin{equation*}
  \TC := \bigl\{ \bar{x} \in X^{\Z_+}\ :\ \exists w_t \in W \mbox{ with } x_{t+1} \equiv f(x_t,w_t),\ t\in\Z_+ \bigr\}.%
\end{equation*}
\end{lemma}

\begin{IEEEproof}
We consider the map%
\begin{equation*}
  G:X \tm W^{\Z_+} \rightarrow X^{\Z_+},%
\end{equation*}
which maps a pair $(x_0,\bar{w})$ with $\bar{w} = (w_t)_{t\in\Z_+}$ to the trajectory $(x_t)_{t\in\Z_+}$ obtained by $x_{t+1} := f(x_t,w_t)$. We claim that this map is measurable and its associated push-forward operator on measures maps $\pi \tm \nu^{\Z_+}$ to $\mu$. To prove that $G$ is measurable, consider a cylinder set $A = B_0 \tm \cdots \tm B_n \tm X^{[n+1,\infty)}$ in $X^{\Z_+}$. Then%
\begin{align*}
  &G^{-1}(A) \\ 
  &= \left\{ (x_0,\bar{w}) : x_0 \in B_0, G(x_0,\bar{w})_1 \in B_1,\ldots, G(x_0,\bar{w})_n \in B_n \right\}.%
\end{align*}
Hence, $G^{-1}(A)$ can be expressed as the preimage of $B_0 \tm \cdots \tm B_n \subset X^{n+1}$ under the map%
\begin{equation*}
  (x_0,\bar{w}) \mapsto (x_0,f_{w_0}(x_0),f_{w_1} \circ f_{w_0}(x_0),\ldots,f_{w_{n-1}} \circ \cdots \circ f_{w_0}(x_0)).%
\end{equation*}
To show that this map is measurable, it suffices to show that each component is a measurable map. This follows from the fact that the projection $W^{\Z_+} \rightarrow W^{n+1}$ to the first $n+1$ components is measurable and $f$ is measurable. Hence, we have proved that $G$ is measurable. To see that $G_*(\pi \tm \nu^{\Z_+}) = \mu$, observe that for a set of the form $A = B_0 \tm B_1 \tm X^{[2,\infty)}$ we have%
\begin{eqnarray}
  && \pi \tm \nu^{\Z_+}(\{(x_0,\bar{w}) : x_0 \in B_0, f_{w_0}(x_0) \in B_1\}) \nonumber \\
&&  = \int_X \int_{W^{\Z_+}}\nu^{\Z_+}(\rmd\bar{w}) \pi(\rmd x_0) \unit_{B_0}(x_0) \unit_{B_1}(f_{w_0}(x_0)) \nonumber\\
&& = \int_{B_0} \pi(\rmd x_0) \int_W \nu(\rmd w) \unit_{B_1}(f_w(x_0)) \nonumber \\
&&= \int_{B_0} \pi(\rmd x_0) \nu(\{ w\in W : f_w(x_0) \in B_1 \}) \nonumber\\
&&= \mu(B_0 \tm B_1 \tm X^{[2,\infty)}).%
\end{eqnarray}
For more general cylinder sets, the claim follows inductively. The fact that $\supp\mu$ is contained in $\cl\TC$ follows from%
\begin{align*}
  \mu(\cl \TC) &= G_*[\pi \tm \nu^{\Z_+}](\cl \TC) = \pi \tm \nu^{\Z_+}(G^{-1}(\cl\TC))\\
	             &   \geq \pi \tm \nu^{\Z_+}(G^{-1}(\TC))\\
	             &= \pi \tm \nu^{\Z_+}(G^{-1}(G(X \tm W^{\Z_+}))) \\
	             & \pi \tm \nu^{\Z_+}(X \tm W^{\Z_+}) = 1.%
\end{align*}
Finally, we show that $\mu$ is $\theta$-invariant. To this end, note that the map $\Phi:X \tm W^{\Z_+} \rightarrow X \tm W^{\Z_+}$, $(x,\bar{w}) \mapsto (f(x,w_0),\vartheta\bar{w})$, satisfies $\theta \circ G = G \circ \Phi$. Using that%
\begin{align*}
  &\pi \tm \nu^{\Z_+}(\Phi^{-1}(A \tm B))\\
   &= \pi \tm \nu^{\Z_+}(\{(x_0,\bar{w}) : f_{w_0}(x_0) \in A,\ \vartheta\bar{w} \in B\})\\
	&= \pi \tm \nu^{\Z_+}\Bigl(\bigcup_{x_0\in X} \{x_0\} \tm ( (f^{x_0})^{-1}(A) \tm B ) \Bigr)\\
  &= \int_X \pi(\rmd x_0) \nu (\{ w: f(x_0,w) \in A \}) \nu^{\Z_+}(B) \\
	&= \nu^{\Z_+}(B) \int_X \pi(\rmd x) P(x,A)\\
	&= \pi(A) \nu^{\Z_+}(B) = \pi \tm \nu^{\Z_+}(A \tm B),%
\end{align*}
i.e., $\Phi_* (\pi \tm \nu^{\Z_+}) = \pi \tm \nu^{\Z_+}$, we find that%
\begin{equation*}
  \theta_*\mu = \theta_* G_* (\pi \tm \nu^{\Z_+}) = G_* \Phi_* (\pi \tm \nu^{\Z_+}) = G_* (\pi \tm \nu^{\Z_+}) = \mu,%
\end{equation*}
completing the proof.%
\end{IEEEproof}

We will also need the following characterization of topological entropy.%

\begin{lemma}\label{lem_entr_timeshift}
Let $f:X\rightarrow X$ be a homeomorphism on a compact metric space $(X,d)$. Fix $\ep>0$ and $n_0 \in \N$. For $n>n_0$ we say that a set $E \subset X$ is $(n,\ep;n_0)$-separated if $d(f^i(x),f^i(y)) > \ep$ for some $i \in \{n_0,n_0+1,\ldots,n-1\}$, whenever $x,y\in E$ with $x\neq y$. We write $r_{\sep}(n,\ep;n_0,f)$ for the maximal cardinality of an $(n,\ep;n_0)$-separated set. Then, for any choice of $n_0(\ep) \in \N$, $\ep>0$, we have%
\begin{equation}\label{eq_entr_timeshift}
  h_{\tp}(f) = \lim_{\ep\downarrow0}\limsup_{n_0(\ep)<n\rightarrow\infty}\frac{1}{n}\log r_{\sep}(n,\ep;n_0(\ep)).%
\end{equation}
\end{lemma}

\begin{IEEEproof}
Any $(n,\ep;n_0(\ep))$-separated set is trivially $(n,\ep)$-separated, hence $r_{\sep}(n,\ep) \geq r_{\sep}(n,\ep;n_0(\ep))$, implying the inequality ``$\geq$'' in \eqref{eq_entr_timeshift}. Conversely, assume that $E$ is $(n,\ep)$-separated and put $E' := f^{-n_0(\ep)}(E)$. Then $|E'| = |E|$ and $E'$ is $(n_0(\ep)+n,\ep;n_0(\ep))$-separated. This implies%
\begin{eqnarray*}
&&  \frac{n+n_0(\ep)}{n}\frac{1}{n+n_0(\ep)}\log r_{\sep}(n_0(\ep)+n,\ep;n_0(\ep)) \\
&& \quad \quad  \geq \frac{1}{n}\log r_{\sep}(n,\ep).%
\end{eqnarray*}
Letting $n\rightarrow\infty$ on both sides, we find that%
\begin{equation*}
  \limsup_{n_0(\ep)<n\rightarrow\infty}\frac{1}{n}\log r_{\sep}(n,\ep;n_0(\ep)) \geq h_{\sep}(f,\ep).%
\end{equation*}
Finally, letting $\ep\downarrow0$, the desired inequality follows.%
\end{IEEEproof}

In the following, we assume that the channel is noiseless. In particular, its capacity is given by $C = \log|\MC|$, where $\MC = \MC'$ is the coding alphabet. We will derive lower bounds on $C_0$ for the objectives (E1)--(E3).%

\begin{theorem}\label{thm_topent_lb}
Consider the estimation objective (E1) for an initial measure $\pi_0$ which is stationary under the Markov chain $(x_t)_{t\in\Z_+}$. If $\supp\mu$ is not compact, we have $C_0 = \infty$. Otherwise,%
\begin{equation*}
  C_0 \geq h_{\tp}(\theta_{|\supp\mu}).%
\end{equation*}
\end{theorem}

\begin{IEEEproof}
Assume that for some $\ep > 0$ the objective \eqref{eq_BowenBall_obj} is achieved by a coder-estimator pair via a noiseless channel of capacity $C = \log|\MC|$. Then for every $k\in\N$ we define the set%
\begin{equation*}
  \EC_k := \left\{ (\hat{x}_0,\hat{x}_1,\ldots,\hat{x}_{k-1}) : q_t \in \MC,\ 0 \leq t \leq k-1 \right\}%
\end{equation*}
of all possible estimation sequences of length $k$ the estimator can generate in the time interval $[0,k-1]$.%

Assume to the contrary that there exists a measurable set $A \subset X^{\Z_+}$ of positive measure $\alpha := \mu(A) > 0$ so that for every $\bar{x} = (x_t)_{t\in\Z_+} \in A$ there is $t \geq T(\ep)$ with $d(x_t,\hat{x}_t) > \ep$ in case the sequence $(x_t)$ is realized as a trajectory of the system. If $G:X\tm W^{\Z_+}\rightarrow X^{\Z_+}$ is the map from the proof of Lemma \ref{lem_mu}, then the preimage $G^{-1}(A)$ is measurable in $X \tm W^{\Z_+}$ with $\pi_0 \tm \nu^{\Z_+}$-measure $\alpha>0$. This contradicts the assumption that the almost sure estimation objective \eqref{eq_BowenBall_obj} is achieved. Hence, the set%
\begin{equation*}
  \tilde{\TC} := \bigl\{ \bar{x} \in X^{\Z_+}\ :\ d(x_t,\hat{x}_t) \leq \ep \mbox{ for all } t \geq T(\ep) \bigr\}%
\end{equation*}
has measure one and consequently is dense in $\supp\mu$.%

Choose $\tau = \tau(\ep)$ large enough so that%
\begin{equation*}
  \sum_{t=\tau}^{\infty}\frac{1}{2^t} \leq \ep.%
\end{equation*}
Let $E \subset \supp\mu$ be a finite $(k,5\ep;T(\ep))$-separated set for some $k > T(\ep)$. Since $\tilde{\TC}$ is dense in $\supp\mu$, a small perturbation of $E$ yields a $(k,5\ep;T(\ep))$-separated set in $\tilde{\TC}$ with the same cardinality as $E$ (using that $\theta$ is continuous). Hence, we may assume $E \subset \tilde{\TC}$. We define a map $\alpha:E \rightarrow \EC_{k+\tau}$ by assigning to $(x_t)_{t\in\Z_+} \in E$ the estimation sequence generated by the estimator when it receives the signals $q_t = q_t(x_0,\ldots,x_t)$ for $t = 0,1,\ldots,k+\tau-1$.%

Assuming $\alpha(\bar{x}) = \alpha(\bar{y})$ for some $\bar{x},\bar{y} \in E$, we find for $T(\ep) \leq t \leq k$ that%
\begin{align*}
  &D(\theta^t(\bar{x}),\theta^t(\bar{y})) \\
  &\leq \sum_{s=0}^{\tau-1}\frac{1}{2^s}\frac{d(x_{t+s},y_{t+s})}{1 + d(x_{t+s},y_{t+s})} + \sum_{s=\tau}^{\infty}\frac{1}{2^s}\\
	                                    &\leq \sum_{s=0}^{\tau-1}\frac{1}{2^s}d(x_{t+s},\hat{x}_{t+s}) + \sum_{s=0}^{\tau-1}\frac{1}{2^s}d(\hat{y}_{t+s},y_{t+s}) + \ep\\
																			&\leq 2\ep + 2\ep + \ep = 5\ep,%
\end{align*}
implying $\bar{x} = \bar{y}$, since $E$ is $(k,5\ep;T(\ep))$-separated. Hence, the map $\alpha$ is injective.%

The set $\supp\mu$ is a closed subset of the complete metric space $(X^{\Z_+},D)$, hence it is also a complete metric space. If we assume that $\supp\mu$ is not compact, it thus follows that $\supp\mu$ is not totally bounded, implying that $(k,5\ep;T(\ep))$-separated subsets of $\supp\mu$ of arbitrarily large (finite) cardinality exist. Hence, $\EC_{k+\tau}$ must be infinite, leading to the contradiction $|\MC| = \infty$. Consequently, in this case the estimation problem cannot be solved via a channel of finite capacity.%

Now assume that $\supp\mu$ is compact. Choosing a maximal $(k,5\ep;T(\ep))$-separated set $E$, for the dynamical system $\theta_{|\supp\mu}:\supp\mu \rightarrow \supp\mu$ we obtain the inequality%
\begin{equation*}
  r_{\sep}(k,5\ep;T(\ep)) \leq |\EC_{k+\tau}| \leq |\MC|^{k+\tau}.%
\end{equation*}
This implies%
\begin{equation*}
  \limsup_{k\rightarrow\infty}\frac{1}{k}\log r_{\sep}(k,5\ep;T(\ep)) \leq \log|\MC| = C.%
\end{equation*}
Using Lemma \ref{lem_entr_timeshift}, the result follows by letting $C\rightarrow C_{\ep}$ and $\ep\downarrow0$. The statement about the metric entropy now follows from the variational principle.%
\end{IEEEproof}

\begin{remark}
To make the statement of the theorem clearer, let us consider the two extreme cases when there is no noise and when there is only noise:%
\begin{enumerate}
\item[(i)] If the system is deterministic, i.e., $x_{t+1} = f(x_t)$ for a homeomorphism $f:X\rightarrow X$ of a compact metric space $X$, then $\pi_0$ is an invariant measure of $f$. Moreover, $P(x_t \in B|x_{t-1} = x) = 1$ if $f(x) \in B$ and $0$ otherwise, implying%
\begin{align*}
  & \mu(B_0 \tm B_1 \tm \cdots \tm B_n \tm X^{[n+1,\infty)})\\
	&= \pi_0\bigl(B_0 \cap f^{-1}(B_1) \cap f^{-2}(B_2) \cap \ldots \cap f^{-n}(B_n)\bigr).%
\end{align*}
From this expression, we see that the support of $\mu$ is contained in the set $\TC$ of all trajectories of $f$ (which in this case coincides with its closure), as already proved in Lemma \ref{lem_mu}. The map $h:\TC \rightarrow X$ defined by $h(\bar{x}) := x_0$, is easily seen to be a homeomorphism, which conjugates $\theta_{|\TC}$ and $f$. That is, the following diagram commutes:%
\begin{equation*}
  \begin{CD}
    \TC @>\theta>> \TC\\
    @V h VV @VV h V\\
    X @>>f> X
  \end{CD}
\end{equation*}
Since $h_*\mu = \pi_0$ and conjugate systems have the same entropy, our theorem implies%
\begin{equation}\label{eq_detlbe1}
  C_0 \geq h_{\tp}(f;\supp\pi_0).%
\end{equation}
The right-hand side of this inequality is finite under mild assumptions, e.g., if $f$ is Lipschitz continuous on $\supp\pi_0$ and $\supp\pi_0$ has finite lower box dimension (see \cite[Thm.~6.1.2]{Boichenko_etal}). These conditions are in particular satisfied when $f$ is a diffeomorphism on a finite-dimensional manifold. However, one should be aware that even on a compact interval there exist continuous maps with infinite topological entropy on the support of an invariant measure. The lower bound \eqref{eq_detlbe1} has already been derived in \cite[Thm.~III.1]{KawanYukselITarXiv}, and in fact for the deterministic case considered here the bound was shown to be tight.%
\item[(ii)] Assume that $X = W$ is compact and the system is given by $x_{t+1} = w_t$, i.e., the trajectories are only determined by the noise. In this case, with $\pi_0 := \nu$, the measure $\mu$ is the product measure $\nu^{\Z_+}$. Hence, $C_0$ is bounded below by the topological entropy of the shift on $W^{\Z_+}$ restricted to $\supp\nu^{\Z_+} = (\supp\nu)^{\Z_+}$. This number is finite if and only if $\supp\nu$ is finite and in this case is given by $\log|\supp\nu|$. 
\end{enumerate}
\end{remark}

If the system is not deterministic, then usually $C_0 = \infty$. In fact, this is always the case if the estimator is able to recover the noise to a sufficiently large extent. The following corollary treats the case, when the noise can be recovered completely from the state trajectory.%

\begin{corollary}\label{cor_inv}
Additionally to the assumptions in Theorem \ref{thm_topent_lb}, suppose that $W$ and $X$ are compact and $f^x:W \rightarrow X$ is invertible for every $x\in X$ so that $(x,y) \mapsto (f^x)^{-1}(y)$ is continuous. Then, for (E1),%
\begin{equation}\label{eq_c0_ilb}
  C_0 \geq h_{\tp}(\Phi_{|\supp(\pi_0 \tm \nu^{\Z_+})}) \geq h_{\tp}(\vartheta_{|\supp\nu^{\Z_+}}),%
\end{equation}
where $\Phi:X \tm W^{\Z_+} \rightarrow X \tm W^{\Z_+}$ is the skew-product map $(x,\bar{w}) \mapsto (f_{w_0}(x),\vartheta\bar{w})$. As a consequence, $C_0 = \infty$ whenever $\supp\nu$ contains infinitely many elements.%
\end{corollary}

\begin{IEEEproof}
We consider the map $h:X^{\Z_+} \rightarrow W^{\Z_+}$, $\bar{x} \mapsto \bar{w} = (w_t)_{t\in\Z_+}$ with%
\begin{equation*}
  w_t = (f^{x_t})^{-1}(x_{t+1}).%
\end{equation*}
If we equip $W^{\Z_+}$ with the product topology, $h$ becomes continuous. Indeed, if the distance of two points $\bar{x}^1,\bar{x}^2 \in X^{\Z_+}$ is small, then the distances $d_X(\bar{x}^1_t,\bar{x}^2_t)$ are small for finitely many values of $t$. Hence, by the uniform continuity of $(x,y) \mapsto f_x^{-1}(y)$ on the compact space $X\tm X$, also the distances $d_W(h(\bar{x}^1)_t,h(\bar{x}^2)_t)$ can be made small for sufficiently many values of $t$, guaranteeing that $D(h(\bar{x}^1),h(\bar{x}^2))$ becomes small, where $D$ is a product metric on $W^{\Z_+}$.%

The map $G:X\tm W^{\Z_+} \rightarrow X^{\Z_+}$, used in the proof of Lemma \ref{lem_mu}, satisfies%
\begin{equation*}
  h(G(x_0,\bar{w})) = \bar{w} \mbox{\quad for all\ } (x_0,\bar{x}) \in X \tm W^{\Z_+},%
\end{equation*}
because we can write%
\begin{equation*}
  G(x_0,\bar{w}) = (x_0,f^{x_0}(w_0),f^{x_1}(w_1),f^{x_2}(w_2),\ldots).%
\end{equation*}
Consequently, $G$ - as a map from $X \tm W^{\Z_+}$ to the space $\TC$ of trajectories -  is invertible with%
\begin{equation*}
  G^{-1}(\bar{x}) = (x_0,h(\bar{x})).%
\end{equation*}
From the assumptions it follows that $G$ is continuous, hence $G$ is a homeomorphism and $\TC$ is compact. By the proof of Lemma \ref{lem_mu}, we have $\theta \circ G = G \circ \Phi$, where $\Phi$ is the skew-product map $\Phi(x,\bar{w}) = (f(x,w_0),\vartheta\bar{w})$ and $G_* (\pi_0 \tm \nu^{\Z_+}) = \mu$. Hence, $G$ is a topological conjugacy between $\theta_{|\supp\mu}$ and $\Phi_{|\supp(\pi_0\tm\nu^{\Z_+})}$, implying%
\begin{equation*}
  C_0 \geq h_{\tp}(\theta_{|\supp\mu}) = h_{\tp}(\Phi_{|\supp(\pi_0 \tm \nu^{\Z_+})}).%
\end{equation*}
Since the projection map $\pi:(x,\bar{w}) \mapsto \bar{w}$ exhibits $\vartheta$ as a topological factor of $\Phi$ and $\pi_* (\pi_0 \tm \nu^{\Z_+}) = \nu^{\Z_+}$, the second inequality in \eqref{eq_c0_ilb} follows.%
\end{IEEEproof}

\begin{example}
Let $X = W = S^1 = \R / \Z$. Let $f(x,w) = x + w \mbox{ mod } 1$ and let $\pi_0 = \nu$ be the normalized Lebesgue measure on $S^1$. In this case, the map $f^x:S^1 \rightarrow S^1$, $w \mapsto x + w$, is obviously invertible and $(x,y) \mapsto (f^x)^{-1}(y) = x - y$ is continuous. Hence, $C_0 = \infty$ for the estimation objective (E1).%
\end{example}

\begin{theorem}\label{thm_metricent_lb_det}
Consider the estimation objective (E2) for an initial measure $\pi_0$ which is stationary and ergodic under the Markov chain $(x_t)_{t\in\Z_+}$. Then, if $\supp\mu$ is compact,
\begin{equation*}
  C_0 \geq h_{\mu}(\theta).%
\end{equation*}
\end{theorem}

\begin{IEEEproof}
First observe that the ergodicity of $\pi_0$ implies the ergodicity of $\mu$. Indeed, it is well-known that the product measure $\pi_0 \tm \nu^{\Z_+}$ is ergodic for the skew-product $\Phi$ if $\pi_0$ is ergodic (cf.~\cite{LedrappierYoung}). Since $G_*(\pi_0 \tm \nu^{\Z_*}) = \mu$ and $\theta \circ G = G \circ \Phi$, this implies the ergodicity of $\mu$. Now consider a noiseless channel with input alphabet $\MC$ and a pair of coder and decoder/estimator which solves the estimation problem (E2) for some $\ep>0$. For every $\bar{x} \in X^{\Z_+}$ and $\delta>\ep$ let%
\begin{equation*}
  T(\bar{x},\delta) := \inf\Bigl\{ k\in\N\ :\ \sup_{t\geq k}d(x_t,\hat{x}_t) \leq \delta \Bigr\},%
\end{equation*}
where the infimum is defined as $+\infty$ if the corresponding set is empty. Note that $T(\bar{x},\delta)$ depends measurably on $\bar{x}$. Define%
\begin{equation*}
  B^K(\delta) := \left\{ \bar{x} \in \supp\mu\ :\ T(\bar{x},\delta) \leq K \right\} \forall \delta > \ep, K \in \N,%
\end{equation*}
and observe that these sets are measurable. From \eqref{eq_AlmostSureBowenBall_obj} it follows that for every $\delta>\ep$,%
\begin{equation*}
  \lim_{K\rightarrow\infty}\mu(B^K(\delta)) = \mu\Bigl(\bigcup_{K\in\N}B^K(\delta)\Bigr) = 1.%
\end{equation*}
Fixing a $K$ large enough so that $\mu(B^K(\delta)) > 0$, Katok's characterization of metric entropy yields the assertion, which is proved with the same arguments as in the proof of Theorem \ref{thm_topent_lb}, using the simple fact a maximal $(n,\ep)$-separated set contained in some set $K$ also $(n,\ep)$-spans $K$.%
\end{IEEEproof}

In the following, we consider (E3). To obtain a lower bound, we restrict the encoder to have finite memory and be periodic.%

\begin{theorem}\label{thm_metricent_lb_detn}
Consider the estimation objective (E3) for an initial measure $\pi_0$ which is stationary and ergodic under the Markov chain $(x_t)_{t\in\Z_+}$. Additionally, assume that there exists $\tau>0$ so that the coder map $\delta_t$ is of the form%
\begin{equation}\label{constraintE}
  q_t = \delta_t(x_{[t-\tau+1,t]})%
\end{equation}
and is periodic so that $\delta_{t + \tau} \equiv \delta_t$. Further assume that the estimator map is of the form%
\begin{equation*}
  \hat{x}_t = \gamma_t(q_{[t-\tau+1,t]})%
\end{equation*}
and also $\gamma_{t+\tau} \equiv \gamma_t$. Then, if $\supp\mu$ is compact, the smallest channel capacity above which (E3) can be achieved for every $\ep>0$ satisfies
\begin{equation*}
  C_0 \geq h_{\mu}(\theta).%
\end{equation*}
\end{theorem}

\begin{IEEEproof}
{\bf Step (i).} First note that we would obtain a lower bound on $C_0$ if we allowed the periodic encoders to be of the form:
\begin{equation}\label{constraintE2}
  q_t = \delta_t(x_{[t-\tau+1,\infty)})
\end{equation}
that is, if we allow the encoder to have non-causal access to the realizations of $x_t$. Note that every encoder policy of the form (\ref{constraintE}) would be of the form (\ref{constraintE2}). We keep the structure of the decoder as is. 

{\bf Step (ii).} The criterion (E3) considered in this paper implies (E3) considered in \cite{KawanYukselITarXiv} with the distortion metric $d$ being the product metric $D$ introduced in (\ref{productM}) for the dynamical system $\theta$ and $p=2$: This follows since if $\limsup_{t \to \infty}  E[(d(x_{t}, \hat{x}_{t})^2 ] \leq \ep$, we have that with $\bar{x}_t = (x_t, x_{t+1},\ldots)$ and $\hat{\bar{x}}_t = (\hat{x}_t,\hat{x}_{t+1},\ldots)$,
\begin{eqnarray}
&&\limsup_{t \to \infty} E[D(\bar{x}_t, \hat{\bar{x}}_t)^2] \nonumber \\
&&=\limsup_{t \to \infty} E\bigg[\bigg(\sum_{i=0}^{\infty} 2^{-i} { d(x_{t+i}, \hat{x}_{t+i}) \over 1+ d(x_{t+i}, \hat{x}_{t+i})}\bigg)^2\bigg]  \nonumber \\
&&\leq \limsup_{t \to \infty} E\bigg[4 \bigg(\sum_{i=0}^{\infty} 2^{-{(i+1)}} { d(x_{t+i}, \hat{x}_{t+i}) \over 1+ d(x_{t+i}, \hat{x}_{t+i})}\bigg)^2\bigg] \nonumber \\
&&\leq  4 \limsup_{t \to \infty} \sum_{i=0}^{\infty} 2^{-{(i+1)}} E\bigg[ \bigg({ d(x_{t+i}, \hat{x}_{t+i}) \over 1+ d(x_{t+i}, \hat{x}_{t+i})}\bigg)^2\bigg] \label{Jensen1} \nonumber\\
&&\leq  4\sum_{i=0}^{\infty} 2^{-{(i+1)}} \limsup_{t \to \infty}  E\bigg[ \bigg({ d(x_{t+i}, \hat{x}_{t+i}) \over 1+ d(x_{t+i}, \hat{x}_{t+i})}\bigg)^2\bigg] \nonumber \\
&&\leq  4\sum_{i=0}^{\infty} 2^{-{(i+1)}} \limsup_{t \to \infty}  E\bigg[ \bigg({ d(x_{t+i}, \hat{x}_{t+i})^2 \over 1}\bigg)\bigg] \nonumber \\
&&=  4\sum_{i=0}^{\infty}2^{-{(i+1)}} \limsup_{t \to \infty}  E[(d(x_{t+i}, \hat{x}_{t+i})^2 ] \nonumber \\
&&\leq  4\sum_{i=0}^{\infty} 2^{-{(i+1)}} \ep \nonumber\\
&&= 4\ep =: \bar{\ep}.
\end{eqnarray}
In particular, $\bar{\ep} \to 0$ as $\ep \to 0$. Thus, if (E3) holds for every $\ep>0$, (E3) considered in \cite{KawanYukselITarXiv} also holds for every $\ep>0$. Here, we apply Jensen's inequality in (\ref{Jensen1}).%

{\bf Step (iii).} Thus, if the encoder is of the form (\ref{constraintE2}), the problem can be viewed as an instance of  \cite[Thm.~V.2]{KawanYukselITarXiv} for the dynamical system $\theta$.

Under the stated periodicity assumption and  (\ref{constraintE2}), \cite[Thm.~V.2]{KawanYukselITarXiv} directly implies that $C_0 \geq h_{\mu}(\theta)$.
\end{IEEEproof}

Three remarks are in order.

\begin{remark}
It is worth noting here that for a deterministic dynamical system, the property of being ergodic is typically very restrictive; however, for a stochastic system ergodicity is often very simple to satisfy: The presence of noise often leads to strong mixing conditions which directly leads to ergodicity.
\end{remark}

\begin{remark}\label{connectionsInfoT} The discussion in Theorem \ref{thm_metricent_lb_detn} leads to an interesting relation between the classical information-theoretic problem of optimally encoding (non-causally) sequences of random variables and metric entropy of an infinite-dimensional dynamical system defined via the shift operator. A close look at the proof of \cite[Thm.~V.1]{KawanYukselITarXiv} reveals that under a stationarity and ergodicity assumption, when the channel is noisefree, the lower bound presented in Theorem \ref{thm_metricent_lb_detn} is essentially achievable, provided that the encoder has non-causal access to the source realizations and in particular a large enough {\it look-ahead} is sufficient for an approximately optimal performance. Note though that the decoder is still restricted to be zero-delay.
\end{remark}

\begin{remark}\label{SavkinRemark}
Besides our results in \cite{KawanYukselITarXiv}, a related study to the approach of this section is due to Savkin \cite{Savkin06}. This paper is concerned with nonlinear systems of the form%
\begin{equation*}
  x(t) = F(x(t),\omega(t))%
\end{equation*}
where $\omega(t)$ is interpreted as an uncertainty input or a disturbance. However, no statistical structure is imposed on $\omega$ so that the system can be regarded as deterministic (thus, the formulation is distribution-free). A characterization of the smallest channel capacity $C_0$ above which the state $x(t)$ can be estimated with arbitrary precision and for every initial state $x(0)$ in a specified compact set via a noiseless channel is given by \cite[Thm.~3.1]{Savkin06}. A close inspection of this result shows that it characterizes $C_0$ precisely as the topological entropy of the associated shift operator acting on system trajectories. Moreover, \cite[Thm.~3.2]{Savkin06} shows that under mild assumptions on the system the entropy of this operator is infinite, and hence observation via a finite-capacity channel is not possible.%
\end{remark}

\section{An information-theoretic and probability-theoretic approach}\label{sec_IT}

In this section, we consider a much broader class of channels, so-called \emph{Class A type channels} (see \cite[Def.~8.5.1]{YukselBasarBook}). We restrict ourselves to systems with state space $X = \R^N$ and provide lower bounds of the channel capacity for the objectives (E2) and (E3), using information-theoretic methods.%

\begin{definition}\label{ClassAChannelDefinition}
A channel is said to be of Class A type, if%
\begin{itemize}
\item it satisfies the following Markov chain condition:%
\begin{equation*}
  q'_t \leftrightarrow q_t, q_{[0,t-1]}, q'_{[0,t-1]} \leftrightarrow \{x_{0}, w_s, s \geq 0\},\label{classAMarkov}
\end{equation*}
i.e., almost surely, for all Borel sets $B$,%
\begin{eqnarray*}
&&  P(q'_t \in B| q_t, q_{[0,t-1]}, q'_{[0,t-1]}, x_{0}, w_s, s \geq 0) \\
&&  = P(q'_t \in B| q_t, q_{[0,t-1]}, q'_{[0,t-1]})%
\end{eqnarray*}
for all $t \geq 0$, and%
\item its capacity with feedback is given by%
\begin{eqnarray*}
&&  C = \lim_{T \to \infty} \max_{\{ P(q_t | q_{[0,t-1]},q'_{[0,t-1]}),\ 0 \leq t \leq T-1\}} \\
&&\quad \quad \quad \quad  \frac{1}{T} I(q_{[0,T-1]} \to q'_{[0,T-1]}),%
\end{eqnarray*}
where the directed mutual information is defined by%
\begin{equation*}
  I(q_{[0,T-1]} \to q'_{[0,T-1]}) := \sum_{t=1}^{T-1} I(q_{[0,t]};q'_t|q'_{[0,t-1]}) + I(q_0;q'_0)%
\end{equation*}
\end{itemize}
\end{definition}

Discrete noiseless channels and memoryless channels belong to this class; for such channels, feedback does not increase the capacity \cite{Cover}. Class A type channels also include finite state stationary Markov channels which are indecomposable \cite{PermuterWeissmanGoldsmith}, and non-Markov channels which satisfy certain symmetry properties \cite{SenAlaYukIT}. Further examples can be found in \cite{DaboraGoldsmith,TatikondaIT}.%

\begin{theorem}\label{thm_ITnegResult1}
Consider system \eqref{eq_stochsys} with state space $X = \R^N$. Suppose that%
\begin{equation*}
  \limsup_{T \to \infty}\frac{1}{T}\sum_{t=1}^{T-1}h(x_t|x_{t-1}) > -\infty%
\end{equation*}
and $h(x_t) < \infty$ for all $t \in \Z_+$. Then, under (E3) (and thus under (E1)),%
\begin{equation*}
  C_{\ep} \geq \bigg(\limsup_{T \to \infty} \frac{1}{T} \sum_{t=1}^{T-1} h(x_t|x_{t-1})\bigg) - \frac{N}{2}\log(2 \pi \rme \ep).%
\end{equation*}
In particular, $C_0 = \infty$.%
\end{theorem}

\begin{IEEEproof}
Let $(\ep_t)_{t\in\Z_+}$ be a sequence of non-negative real numbers so that $E[\|x_t - \hat{x}_t\|^2] \leq \ep_t$ for all $t \in \Z_+$ and $\limsup_{t\rightarrow\infty}\ep_t \leq \ep$. Observe that for every $t > 0$ we have%
\begin{eqnarray}\label{usefulStep}
&&  I(q'_t;q_{[0,t]}|q'_{[0,t-1]}) \nonumber \\
&=& H(q'_t|q'_{[0,t-1]}) -  H(q'_t|q_{[0,t]},q'_{[0,t-1]}) \nonumber \\
&=& H(q'_t|q'_{[0,t-1]}) -  H(q'_t|q_{[0,t]},x_t,q'_{[0,t-1]}) \label{classAassumption} \\
&\geq& H(q'_t|q'_{[0,t-1]}) -  H(q'_t|x_t,q'_{[0,t-1]}) \nonumber \\
&=& I(x_t; q'_t|q'_{[0,t-1]}). \nonumber
\end{eqnarray}
Here, \eqref{classAassumption} follows from the assumption that the channel is of Class A type. Define%
\begin{equation*}
  R_T := \max_{\{ P(q_t | q_{[0,t-1]},q'_{[0,t-1]}),\ 0 \leq t \leq T-1\}} \frac{1}{T}\sum_{t=0}^{T-1} I(q'_t; q_{[0,t]} | q'_{[0,t-1]})%
\end{equation*}
Now consider the following identities and inequalities:%
\begin{eqnarray}
&&\lim_{T \to \infty} R_T\nonumber \allowdisplaybreaks \\
&\geq& \limsup_{T \to \infty} \frac{1}{T} \bigg( \sum_{t=1}^{T-1} I(x_t; q'_t| q'_{[0,t-1]})) + I(x_0;q'_0) \bigg) \nonumber \allowdisplaybreaks \\
 &=& \limsup_{T \to \infty} \frac{1}{T}\sum_{t=1}^{T-1} \bigg(h(x_t | q'_{[0,t-1]}) - h(x_t|q'_{[0,t]})\bigg)\nonumber \allowdisplaybreaks \\
 &\geq & \limsup_{T \to \infty} \frac{1}{T}\sum_{t=1}^{T-1} \bigg(h(x_t | x_{[0,t-1]},q'_{[0,t-1]}) - h(x_t|q'_{[0,t]})\bigg)\nonumber \allowdisplaybreaks\\
 &=& \limsup_{T \to \infty} \frac{1}{T}\sum_{t=1}^{T-1} \bigg(h(x_t | x_{[0,t-1]},q'_{[0,t-1]}) - h(x_t - \hat{x}_t|q'_{[0,t]})\bigg)\nonumber \allowdisplaybreaks\\
 &\geq& \limsup_{T \to \infty} \frac{1}{T}\sum_{t=1}^{T-1} \bigg(h(x_t | x_{[0,t-1]},q'_{[0,t-1]}) - h(x_t - \hat{x}_t)\bigg)\nonumber \allowdisplaybreaks\\
 &\geq& \limsup_{T \to \infty} \frac{1}{T}\sum_{t=1}^{T-1} \bigg(h(x_t | x_{[0,t-1]},q'_{[0,t-1]}) - \frac{N}{2}\log(2\pi e \ep_t)\bigg)\nonumber \allowdisplaybreaks\\
 &=& \limsup_{T \to \infty} \frac{1}{T}\bigg( \sum_{t=1}^{T-1} h(x_t | x_{t-1}) \bigg) - \frac{N}{2}\log(2\pi \rme \ep).\label{eq_denklem}%
\end{eqnarray}
Here, the second inequality uses the property that entropy decreases under conditioning on more information. The second equality follows from the fact that $\hat{x}_t$ is a function of $q'_{[0,t]}$, and the last inequality follows from that fact that among all real random variables $X$ that satisfy a given second moment constraint $E[X^2] \leq \ep$, a Gaussian maximizes the entropy and the differential entropy in this case is given by $\frac{1}{2}\log(2\pi \rme \ep)$. Using the fact that for an $n$-dimensional vector $X= [X_1,\ldots,X_n]^T$, $h(X) = h(X_1) + \sum_{i=2}^n h(X_i | X_{[1,i-1]}) \leq \sum_{i=1}^n h(X_i)$, it follows with $E[\|x_t - \hat{x}_t\|^2] \leq \ep_t$ that $h(x_t - \hat{x}_t) \leq \frac{n}{2}\log(2\pi \rme \ep_t)$. The final equality then follows from the fact that conditioned on $x_{t-1}$, $x_t$ and $q'_{[0,t-1]}$ are independent. For the final result, in \eqref{eq_denklem}, taking the limit as $\ep \to 0$, $\log(\ep) \to -\infty$, and $C_0 = \infty$ follows.%
\end{IEEEproof}

\begin{theorem}\label{thm_ITnegResult2}
Suppose that $X \subset \R^N$ and the system given by $x_{t+1} = f(x_t,w_t)$ is so that for all Borel sets $B \subset \R^N$,%
\begin{equation*}
  P(x_{t+1} \in B | x_t=x) \leq K \lambda(B),%
\end{equation*}
where $\lambda$ denotes the Lebesgue measure, $K \in \mathbb{R}_+$, and $w_t$ is an i.i.d.~noise process. Then, for the objective (E2), $C_0 = \infty$. 
\end{theorem}

A special case for the above is a system of the form%
\begin{equation*}
  x_{t+1} = f(x_t) + w_t,%
\end{equation*}
where $w_t \sim \nu$ with the noise measure $\nu$ admitting a bounded density function.%


\begin{IEEEproof}
Given a finite alphabet channel with $|{\cal M}'| < \infty$, for a given time $t>0$ under any encoding and decoding policy, there exists a finite partition of the state space $X$ for encoding $x_t$ leading to $\hat{x}_t$. Thus there exists $\bar{\ep}>0$ so that for all $\ep \in (0,\bar{\ep})$, for each set%
\begin{equation*}
  A_t(q'_0,\ldots,q'_t) := \left\{x \in \R^N : d\left(x,\hat{x}_t(q'_0,\ldots,q'_{t-1},q'_t)\right) \geq \ep \right\},%
\end{equation*}
where $q_0',\ldots,q_t' \in \MC'$, we find that%
\begin{align*}
  &P\bigl(x_t \in A_t(q'_0,\ldots,q'_t) \bigl| x_{[0,t-1]}, q'_{[0,t-1]}\bigr)\allowdisplaybreaks \\
  &= \sum_{q' \in \MC'} P\left(q'_t = q' | x_{[0,t-1]}, q'_{[0,t-1]}\right) \allowdisplaybreaks\\
	&\qquad \times P\bigl(x_t \in A_t(q'_0,\ldots,q'_t)\bigl| x_{[0,t-1]}, q'_{[0,t-1]},q'_t=q'\bigr) \allowdisplaybreaks\\
 &\geq 1 - \sum_{q' \in \MC'} P\left(q'_{t}=q' | x_{[0,t-1]}, q'_{[0,t-1]}\right) \allowdisplaybreaks\\
&\times P\bigl(d(x_t,\hat{x}_t(q'_0,\ldots,q'_{t-1},q'))  < \ep \bigl| x_{[0,t-1]}, q'_{[0,t-1]},q'_t=q' \bigr) \allowdisplaybreaks\\
&\geq 1 - |\MC'| K \lambda(B_{\ep}(0)) > 0.%
\end{align*}
This implies that
\begin{equation}\label{eq_boundJ}
  P\bigg(x_t \in A_t(q'_0,\ldots,q'_t) \bigg| x_{[0,t-1]}, q'_{[0,t-1]} \bigg) > 0%
\end{equation}
uniform over all realizations of $x_{[0,t-1]}, q'_{[0,t-1]}$.

Let%
%
\begin{equation*}
  \eta := \sum_{t=1}^{\infty} \unit_{\{x_t \in A_t(q'_0,\ldots,q'_t)\}}.%
\end{equation*}
Our goal is to show that $\eta = \infty$ almost surely, leading to the desired conclusion. 
Let
\[\tau(1) = \min\{t > 0: x_t \in A_t(q'_0,\ldots,q'_t)\}\] 
and for $z > 1, z \in \mathbb{N}$
\[\tau(z) = \min\{t > \tau(z-1): x_t \in A_t(q'_0,\ldots,q'_t)\}.\]
It follows that $P(\tau(1) < \infty)=1$ by a repeated use of \eqref{eq_boundJ}, since the event $\tau(1)=\infty$ would imply that the event (whose probability is lower bounded by \eqref{eq_boundJ}) would be avoided infinitely many times leading to a zero measure. Thus, $P(\eta \geq 1) = 1$. By a repeated use of \eqref{eq_boundJ} and induction if $P(\eta \geq k-1) = 1$, we have that
\begin{align*}
  P(\eta \geq k) &= P(\eta \geq k, \eta \geq k-1)\\
	&= P(\tau(1) < \infty | \FC_{\tau(k-1)} ) P(\eta \geq k-1) = 1,
\end{align*}
where $\FC_{\tau(k-1)}$ is the $\sigma$-field generated by $\{x_s, q'_s\}$ up to time $\tau(k-1)$. Thus, for every $k \in \N$, $P(\eta \geq k)=1$, and it follows by continuity in probability that $P(\eta = \infty) = \lim_{k \to \infty} P(\eta \geq k)= 1$. Hence, for any finite communication rate, almost sure boundedness is not possible for arbitrarily small $\ep > 0$.%
\end{IEEEproof}

\section{Achievability bounds}

\subsection{Coding of deterministic dynamical systems over noisy communication channels}\label{sec_noisefree}

In this section, we show that for a noisefree system a discrete memoryless noisy communication channel is no obstruction for achieving the objectives (E2) and (E3) with finite capacity. More precisely, we prove the following theorem.%

\begin{theorem}\label{thm_upperbound}
Consider a nonlinear deterministic system $x_{t+1} = f(x_t)$ given by a continuous map $f:X\rightarrow X$ on a compact metric space $X$, estimated via a discrete memoryless channel (DMC). Then, for the asymptotic estimation objectives (E2) and (E3), we have
\begin{equation*}
  C_0 \leq h_{\tp}(f).%
\end{equation*}
\end{theorem}

\begin{IEEEproof}
It suffices to prove the result for (E2), since for a compact metric space, (E2) implies (E3); therefore the construction below also applies for the objective (E3).

Without loss of generality, we may assume that $h_{\tp}(f) < \infty$, since otherwise the statement trivially holds. Then it suffices to show that for any $\ep>0$ the estimation objective can be achieved whenever the channel capacity satisfies $C > h_{\tp}(f)$. Since the capacity of a DMC can take any positive value, it follows that $C_{\ep} \leq h_{\tp}(f)$ for every $\ep>0$ and thus $C_0 \leq h_{\tp}(f)$. 

Now, consider a channel with capacity $C > h_{\tp}(f)$ and fix $\ep>0$. Recall that the input alphabet is denoted by $\MC$ and the output alphabet by $\MC'$. By the random coding construction of Shannon \cite{GallagerIT85}, we can achieve a rate $R$ satisfying%
\begin{equation}\label{eq_ratecondition}
  h_{\tp}(f) < R < C%
\end{equation}
with a sequence of increasing sets $\{1,\ldots,M_n\}$ of input messages so that for all $n$,%
\begin{equation}\label{eq_rateSeqC}
  2^{nR} \leq M_n \mbox{\quad and\quad } \lim_{ n \to \infty } {1 \over n} \log M_n = C.%
\end{equation}
Furthermore, there exists a sequence of encoders $E^n:\{1,\ldots,M_n\} \rightarrow \MC^n$, yielding codewords $x^n(1),\ldots,x^n(M_n)$, and a sequence of decoders $D^n:(\MC')^n \rightarrow \{1,\ldots,M_n\}$ so that%
\begin{equation*}
  P( D^n(q'_{[0,n-1]}) \neq c| q_{[0,n-1]} = x^n(c) ) \leq \rme^{-nE(R) + o(n)},%
\end{equation*}
uniformly for all $c \in \{1,\ldots,M_n\}$. Here ${o(n) \over n}\to 0$ as $n \to \infty$ and $E(R) > 0$. In particular, we observe that with $c_n \in \{1,\ldots,M_n\}$ being the message transmitted and $D^n(q'_{[0,n-1]})$ the decoder output,%
\begin{align*}
& P( D^n(q'_{[0,n-1]}) \neq c_n ) \nonumber \\
& = \sum_{c  \in \{1,\ldots,M_n\}}P( D^n(q'_{[0,n-1]}) \neq c | q_{[0,n-1]} = x^n(c) )  \nonumber \\
& \qquad \qquad  \qquad \times P(q_{[0,n-1]} = x^n(c))\\
& \leq \rme^{-n E (R) + o(n)}.%
\end{align*}
This also implies that the bound holds even when the messages to be transmitted are not uniformly distributed. Thus, for the sequence of encoders and decoders constructed above we have%
\begin{equation*}
  \sum_n P( D^n(q'_{[0,n-1]}) \neq c_n ) \leq \sum_n \rme^{-n E (R) + o(n)} < \infty.%
\end{equation*}
The Borel-Cantelli Lemma then implies%
\begin{equation}\label{eq_almostSureRecovery}
  P\Bigl(\Bigl\{D^n(q'_{[0,n-1]}) \neq c_n \ \mbox{infinitely often} \Bigr\} \Bigr) = 0.%
\end{equation}
Now we choose $\delta \in (0,\ep)$ so that, by uniform continuity,%
\begin{equation}\label{eq_unco}
  d(x,y) < \delta \quad\Rightarrow\quad d(f(x),f(y)) < \ep \mbox{\quad for all\ } x,y\in X.%
\end{equation}
Furthermore, we choose $N$ sufficiently large so that%
\begin{equation}\label{eq_kchoice}
  r_{\spn}(n,\delta) \leq M_n \mbox{\quad for all\ } n \geq N.%
\end{equation}
This is possible, because by \eqref{eq_ratecondition} and \eqref{eq_rateSeqC}, for every $n \in \N$ we have%
\begin{eqnarray*}
&&  \limsup_{n\rightarrow\infty}\frac{1}{n}\log r_{\spn}(n,\delta) \leq h_{\tp}(f) \\
&&  < R = \frac{1}{n}\log 2^{nR} \leq \frac{1}{n}\log M_n.%
\end{eqnarray*}
Let $S_j$ be a $(j,\ep)$-spanning set of cardinality $r_{\spn}(j,\delta)$ and fix injective functions%
\begin{equation*}
  \iota_j:S_j \rightarrow \{1,\ldots,M_j\}.%
\end{equation*}
In fact, by possibly enlarging the set $S_j$, we can assume that $\iota_j$ is bijective. For any $a\in X$ let $x^*_j(a)$ denote a fixed element of $S_j$ satisfying $d(f^t(x^*(a)),f^t(a)) \leq \delta$ for $0 \leq t \leq j-1$.%

Define sampling times by%
\begin{equation*}
  \tau_0 := 0 \mbox{\quad and\quad} \tau_{j+1} := \tau_j + j+1 \mbox{ for } j\geq0.%
\end{equation*}
In the following, we specify the coding scheme. In this coding scheme, the encoder from $\tau_j$ to $\tau_{j+1}-1$, encodes the information regarding the orbit of the state from $\tau_{j+1}$ to $\tau_{j+2}-1$. For all $j \geq N$, at time $\tau_j$, use the input $\iota_{j+1}(x_{j+1}^*(f^{j+1}(x_{\tau_j})))$ for the encoder, where $x_{\tau_j}$ is the state at time $\tau_j$. Then $x^{j+1}(\iota_{j+1}(x_{j+1}^*(f^{j+1}(x_{\tau_j}))))$ is sent during the next $j+1$ units of time. This is possible by \eqref{eq_kchoice}. For $j<N$, it is not important what we transmit.%

Let the estimator apply $x^*_{j+1} \circ \iota_{j+1}^{-1}$ to the output of the decoder, obtaining an element $y_{j+1} \in S_{j+1}$, and use $y_{j+1},f(y_{j+1}),\ldots,f^j(y_{j+1}),f^{j+1}(y_{j+1})$ as the estimates during the forthcoming time interval of length $\tau_{j+2}-\tau_{j+1} = \tau_{j+1} - \tau_j + 1 = j + 2$. Then $\delta < \ep$, \eqref{eq_unco} and the fact that $S_{j+1}$ is $(j+1,\delta)$-spanning implies that the desired estimation accuracy is achieved, provided that there was no error in the transmission.%

Now \eqref{eq_almostSureRecovery} implies that after a finite random time, there are no more errors in the transmission. By the analysis above, the errors will be uniformly bounded by $\ep$. Hence, the objective \eqref{eq_AlmostSureBowenBall_obj} is achieved.%
\end{IEEEproof}

We note that the proof above crucially depends on the fact that the system is deterministic and is impractical to implement. The theorem is essentially a possibility result. Note that the proof even does not make use of the fact that the encoder has access to the realizations of the channel output, hence feedback is not utilized. For linear systems, a constructive proof is given in \cite[Thm.~6.4.1]{MatveevSavkin}. We state the following for completeness. Consider the noiseless linear system%
\begin{equation}\label{eq_ls}
  x_{t+1} = Ax_t%
\end{equation}
with $x_t \in \R^N$. The following result, essentially given in \cite[Thm.~6.4.1]{MatveevSavkin}, provides a positive answer to the question whether the estimation objective \eqref{eq_AlmostSureBowenBall_obj} can be achieved, when no noise is present in the system.%


\begin{theorem}\label{NoisyAS}
Consider system \eqref{eq_ls} estimated over a memoryless erasure channel with finite capacity. Then, for (E2), we have%
\begin{equation}\label{asStabErasure}
  C_0 \leq \sum_{|\lambda_i|>1} \log\lceil |\lambda_i| \rceil.%
\end{equation}
\end{theorem}

For completeness, we also note that \cite[Cor.~5.3 and Thm.~4.3]{SahaiParts} show that for a discrete memoryless channel it suffices that $C > \sum_{|\lambda_i| \geq 1} \log|\lambda_i|$ for the existence of encoder and controller policies leading to almost sure stability.

\begin{remark}{\bf An implication on the achievability for non-causal codes over discrete memoryless channels.} In Section \ref{sec_dynsys} we utilized the fact that one can view a stochastic dynamical system as a deterministic one under the shift operator. Building on a similar argument as that in the proof of Theorem \ref{thm_metricent_lb_detn}, it can be shown that (E2) considered in this paper implies and is implied by (E2) considered in \cite{KawanYukselITarXiv} with the distortion metric $d$ being the product metric $D$ introduced in (\ref{productM}) for the dynamical system $\theta$. Therefore, provided that the encoder has access to future realizations of the state sequence, the proof of Theorem \ref{thm_upperbound} implies an achievability result: If the encoder has non-causal access to the source realizations, for (E2) it suffices to have $C_0 \leq h_{\tp}(\theta_{|\supp\mu})$ and this can be achieved through the construction in the proof of Theorem \ref{thm_upperbound} through an encoder which has non-causal access to the future state realizations. Note though that the decoder is still restricted to be zero-delay. We note that in the traditional Shannon theory, block codes are allowed to be non-causal. 
\end{remark}

\section{Examples}\label{sec_examples}

\begin{example}
Consider the diffeomorphism $f_A:\T^2 \rightarrow \T^2$ on the $2$-torus $\T^2 = \R^2/\Z^2$, induced by the linear map%
\begin{equation}\label{eq_catmap}
  A = \left(\begin{array}{cc} 2 & 1 \\ 1 & 1 \end{array}\right),%
\end{equation}
i.e., $f_A(x + \Z^2) = Ax + \Z^2$. Note that the inverse of $f_A$ is given by $f_{A^{-1}}$, which is well-defined, since $\det A = 1$. The map $f_A$ is known as \emph{Arnold's Cat Map}, and is one of the simplest examples of an Anosov diffeomorphism.%

Since $\det \rmD f_A(x) \equiv \det A \equiv 1$, the map $f_A$ is area-preserving. The eigenvalues of the matrix $A$ are given by%
\begin{equation*}
  \gamma_1 = -\frac{3}{2} - \frac{1}{2}\sqrt{5} \mbox{\quad and\quad} \gamma_2 = -\frac{3}{2} + \frac{1}{2}\sqrt{5}%
\end{equation*}
and satisfy $|\gamma_1| > 1 > |\gamma_2|$. It is well-known that both the topological entropy and the metric entropy of $f_A$ with respect to Lebesgue measure are given by $\log|\gamma_1| > 0$. Hence, Theorem \ref{thm_upperbound} yields%
\begin{equation*}
  C_0 \leq \log\left|-\frac{3}{2} - \frac{1}{2}\sqrt{5}\right| \approx 1.3885%
\end{equation*}
for (E2) to be achieved over a DMC.%

Now, suppose we have additive noise for the cat map so that $f_A(x + \Z^2) = Ax + w + \Z^2$, with $w \sim \nu$ which admits a density supported on $\T^2$. In this case, the map $f^x:\T^2 \rightarrow \T^2$, $w \mapsto Ax+w$, is invertible and $(x,y) \mapsto (f^x)^{-1}(y) = y -Ax$ is continuous. By Corollary \ref{cor_inv}, $C_0 = \infty$ for the estimation objective (E1), under a stationary initial measure. For the objective (E2) it can be shown that, under corresponding initial measure conditions, Theorem \ref{thm_metricent_lb_det} leads to $C_0 = \infty$.%
\end{example}

\section{Discussion and Concluding Remarks}\label{sec_conclusion}

In this paper, we considered three estimation objectives for stochastic non-linear systems $x_{t+1} = f(x_t,w_t)$ with i.i.d.~noise $(w_t)$, assuming that the estimator receives state information via a noisy channel of finite capacity. 
\begin{enumerate}
\item[(1)] For noiseless channels, assuming that the initial measure $\pi_0$ is stationary, we proved that $C_0$ is bounded below by either the topological or the metric entropy of a shift dynamical system on the space of trajectories (Theorems \ref{thm_topent_lb}, \ref{thm_metricent_lb_det} and \ref{thm_metricent_lb_detn}). 
\item[(2)] For systems on Euclidean space and noisy channels, we provided information-theoretic and probability-theoretic conditions enforcing $C_0 = \infty$. In particular, Theorem \ref{thm_ITnegResult1} shows that $C_0 = \infty$ for the quadratic stability objective, whenever%
\begin{equation}\label{eq_cdecond}
  \limsup_{T \to \infty}\frac{1}{T}\sum_{t=1}^{T-1}h(x_t|x_{t-1}) > -\infty.%
\end{equation}
We have a corresponding negative result under (E2) in Theorem \ref{thm_ITnegResult2} for noisy systems which are sufficiently {\it irreducible}. Since $h(x_t|x_{t-1})$ is a measure for the uncertainty of $x_t$ given $x_{t-1}$, the condition \eqref{eq_cdecond} means that the noise on the long run (in average) influences the state process in a substantial way. Similarly to the results in Section \ref{sec_dynsys}, this means that the noise makes the space of relevant trajectories too large (or too complicated) to estimate the state with arbitrarily small error over a finite capacity channel.%
\item[(3)] Compared with our earlier work \cite{KawanYukselITarXiv}, our results reveal that the rate requirements are not robust with respect to the presence of noise: That is, even an arbitrarily small noise may lead to drastic effects in the rate requirements. However, the metric or topological entropy bounds are always present and our lower bounds reduce to those established in \cite{KawanYukselITarXiv}. We also note that the metric entropy definition for random dynamical systems \cite{LedrappierYoung} in the ergodic theory literature is not the answer to the operational questions we proposed in this paper, unlike the one for the deterministic case which precisely answered the operational question (E2). 
\item[(4)] In Section \ref{sec_noisefree} we assumed that the system is deterministic with a compact state space, but the channel is noisy. We proved that in this case $C_0$ is bounded from above by the topological entropy of the system for the asymptotic almost sure objective (thus, leading to an achievability result). Our result strictly generalizes the previously known results in the literature which have considered only linear systems to our knowledge. 
\end{enumerate}

\end{document}